\documentclass[11pt]{article}

\usepackage{amsmath}
\usepackage{amsthm}
\usepackage{amsfonts}
\usepackage{bbm}

\usepackage[T2A]{fontenc}
\usepackage[cp1251]{inputenc}
\usepackage[english]{babel}

\newcommand{\mmp}{\mathbb{P}}

\newcommand{\tp}{\overset{{\mmp}}{\to}}

\newcommand{\me}{\mathbb{E}}

\newcommand{\mr}{\mathbb{R}}
\newcommand{\mn}{\mathbb{N}}

\renewcommand{\ln}{\log}

\newcommand{\be}{\begin{equation}}
\newcommand{\ee}{\end{equation}}

\DeclareMathOperator{\1}{\mathbbm{1}}

\newtheorem{theorem}{Theorem}[section]
\newtheorem{lemma}[theorem]{Lemma}

\newtheorem{rem}[theorem]{Remark}
\begin{document}

\title{A functional limit theorem for general shot noise processes}
\author{
Alexander Iksanov
    \footnote{Faculty of Computer Science and Cybernetics,
    Taras Shevchenko National University of Kyiv, 01601 Kyiv, Ukraine;
    \ e-mail: iksan@univ.kiev.ua} \ \ and \ \
Bohdan Rashytov
    \footnote{Faculty of Computer Science and Cybernetics,
    Taras Shevchenko National University of Kyiv, 01601 Kyiv, Ukraine;
    \ e-mail: mr.rashytov@gmail.com}}
\maketitle

\begin{abstract}
\noindent By a general shot noise process we mean a shot noise
process in which the counting process of shots is arbitrary
locally finite. Assuming that the counting process of shots
satisfies a functional limit theorem in the Skorokhod space with a
locally H\"{o}lder continuous Gaussian limit process and that the
response function is regularly varying at infinity we prove that
the corresponding general shot noise process satisfies a similar
functional limit theorem with a different limit process and
different normalization and centering functions. For instance, if
the limit process for the counting process of shots is a Brownian
motion, then the limit process for the general shot noise process
is a Riemann-Liouville process. We specialize our result for five
particular counting processes. Also, we investigate H\"{o}lder
continuity of the limit processes for general shot noise
processes.
\end{abstract}

\noindent Keywords: H\"{o}lder continuity; shot noise process; weak
convergence in the Skorokhod space

\noindent 2000 Mathematics Subject Classification: Primary: 60F17 \\
\hphantom{2000 Mathematics Subject Classification: }Secondary: 60G55

\section{Introduction and main result}

Let $(S_k)_{k \in \mn_0}$ be a not necessarily monotone sequence of positive
random variables. Define the counting process $(N(t))_{t \geq 0}$ by
$$ N(t):=\sum_{k \geq 0}\1_{\{S_k \leq t\}}, \quad t \geq 0,$$ where $\1_A=1$ if the event $A$ holds and $=0$, otherwise.
Throughout the paper we always assume that $N(t)<\infty$ almost surely
(a.s.) for $t \geq 0$.

Denote by $D:=D[0,\infty)$ the Skorokhod space of right-continuous
real-valued functions which are defined on $[0,\infty)$ and have
finite limits from the left at each positive point. For a function
$h\in D$, the random process $X:=(X(t))_{t \geq 0}$ which is the
main object of our investigation is given by $$X(t):=\sum_{k \geq
0}{h(t-S_k)}\1_{\{S_k \leq t\}}= \int_{[0,\,t]}{h(t-y)}{\textrm
d}N(y), \quad t \geq 0.$$ We call $X$ {\it general shot noise
process}, for no assumptions are imposed apart from $N(t)<\infty$
a.s. Plainly, $X\in D$ a.s.

Denote by $Y_1:=(Y_1(t))_{t\geq 0}$, $Y_2:=(Y_2(t))_{t\geq
0},\ldots$ independent and identically distributed (i.i.d.) random
processes with paths in $D$. Assume that, for $k\in\mn_0$,
$Y_{k+1}$ is independent of $(S_0,\ldots, S_k)$. In particular,
the case of complete independence of $(Y_j)_{j\in\mn}$ and
$(S_k)_{k\in\mn_0}$ is not excluded. Set $$Y(t):=\sum_{k\geq
0}Y_{k+1}(t-S_k)\1_{\{S_k\leq t\}},\quad t\geq 0$$ and call
$Y:=(Y(t))_{t\geq 0}$ {\it random process with immigration at
random times}. The interpretation is that associated with the
$k$th immigrant arriving at the system at time $S_{k-1}$ is the
process $Y_k$ which defines a model-dependent `characteristic'\,
of the $k$th immigrant. For instance, $Y_k(t-S_{k-1})$ may be the
fitness of the $k$th immigrant at time $t$. The value of $Y(t)$ is
then given by the sum of `characteristics' of all immigrants
arriving at the system up to and including time $t$. Assume that
the function $g(t):=\me Y_1(t)$ is finite for all $t\geq 0$, not
identically $0$ and that $g\in D$. To investigate weak convergence
of the process $Y$, properly normalized and centered, it is
natural to use decomposition
\begin{equation}\label{imm}
Y(t)=\sum_{k\geq 0}(Y_{k+1}(t-S_k)-g(t-S_k))\1_{\{S_k\leq t\}}+\sum_{k\geq
0}g(t-S_k)\1_{\{S_k\leq t\}}, \quad t\geq 0.
\end{equation}
For fixed $t>0$, while the first summand is the terminal value of a martingale, the second is the value at time $t$ of a general shot noise process.
The summands should be treated separately, for each of these requires a specific approach. Weak convergence of the first summand in \eqref{imm} will be
investigated in \cite{Dong+Iksanov:2019}.

In the present paper we are aimed at proving a functional limit
theorem for a general shot noise process $X$ under natural
assumptions. Besides being of independent interest our findings
pave the way towards controlling the asymptotic behavior of the
second summand in \eqref{imm}. These taken together with
prospective results from \cite{Dong+Iksanov:2019} should
eventually lead to understanding of the asymptotics of processes
$Y$.

A rich source of random processes $Y$ with immigration at random
times are queueing systems and various branching processes with or
without immigration. For example, particular instances of random
variables $Y(t)$ are given by the number of the $k$th generation
individuals ($k\geq 2$) with positions $\leq t$ in a branching
random walk; the number of customers served up to and including
time $t$ or the number of busy servers at time $t$ in a
$GEN/G/\infty$ queuing system, where GEN means that the arrival of
customers is regulated by a general point process. Nowadays rather
popular objects of research are queueing systems in which an input
process is more complicated than the renewal process, for
instance, a Cox process (also known as a doubly stochastic Poisson
process) \cite{Chernavskaya:2015} or a Hawkes process
\cite{Daw+Pender:2018,Koops+Saxena+Boxma+Mandjes:2018} and
branching processes with immigration governed by a process which
is more general than the renewal process, for instance, an
inhomogeneous Poisson process \cite{Yakovlev+Yanev:2007} or a Cox
process \cite{Butkovsky:2012}. Note that some authors investigated
the processes $X$, $Y$ or the like from purely mathematical
viewpoint. An incomplete list of relevant publications includes
the works \cite{Daley:1972, Rice:1977, Schmidt:1985,
Westcott:1976} and the recent article \cite{Pang+Zhou:2018}. On
the other hand, we stress that the results obtained in the
aforementioned papers do not overlap with ours.

The present work was preceded by several articles
\cite{Iksanov:2013,Iksanov+Kabluchko+Marynych:2016,
Iksanov+Kabluchko+Marynych+Shevchenko:2016,Iksanov+Marynych+Meiners:2014,
Kabluchko+Marynych:2016} in which weak convergence of renewal shot
noise processes has been investigated. The latter processes is a
particular case of processes $X$ in which the input sequence
$(S_k)_{k\in\mn_0}$ is a standard random walk. Development of
elements of the weak convergence theory for renewal shot noise
processes was motivated by and effectively used for the asymptotic
analysis of various characteristics of several random regenerative
structures: the order of random permutations
\cite{Gnedin+Iksanov+Marynych:2012}, the number of zero and
nonzero blocks of weak random compositions
\cite{Alsmeyer+Iksanov+Marynych:2017,
Iksanov+Marynych+Vatutin:2015}, the number of collisions in
coalescents with multiple collisions
\cite{Gnedin+Iksanov+Marynych+Moehle:2018}, the number of busy
servers in a $G/G/\infty$ queuing system
\cite{Iksanov+Jedidi+Bouzzefour:2018}, random process with
immigration at the epochs of a renewal process
\cite{Iksanov+Marynych+Meiners:2017a,Iksanov+Marynych+Meiners:2017b}.
Chapter 3 of the monograph \cite{Iksanov:2016} provides a survey
of results obtained in the aforementioned articles, pointers to
relevant literature and a detailed discussion of possible
applications.

To formulate our main result we need additional notation. Denote
by $W_\alpha:=(W_\alpha(u))_{u\geq 0}$ a centered Gaussian process
which is a.s.\ locally H\"{o}lder continuous with exponent
$\alpha> 0$ and satisfies $W_\alpha(0)=0$ a.s. In particular, for
all $T>0$, all $0\leq x,y\leq T$ and some a.s.\ finite random
variable $M_T$
\begin{equation}\label{hold}
|W_\alpha(x)-W_\alpha (y)| \leq M_T |x-y|^\alpha.
\end{equation}
Define the random process
$Y_{\alpha,\rho}:=(Y_{\alpha,\rho}(u))_{u\geq 0}$ by
\begin{equation}    \label{defYpos}
Y_{\alpha,\rho}(u):= \rho \int_0^u (u-y)^{\rho-1}
W_{\alpha}(y)\textrm{d}y, \quad u>0,\quad
Y_{\alpha,\rho}(0):=\lim_{u\to +0}Y_{\alpha,\rho}(u),
\end{equation}
when $\rho>0$ and by
\begin{align}
Y_{\alpha,\rho}(u):=u^\rho W_\alpha(u) + |\rho| \int_0^u
(W_{\alpha}(u) - W_{\alpha}(u-y)) y^{\rho-1}\textrm{d}y,  \quad u>
0,\quad Y_{\alpha,\rho}(0):=\lim_{u\to +0}Y_{\alpha,\rho}(u),
\label{defYneg}
\end{align}
when $-\alpha<\rho<0$. Also, put $Y_{\alpha,0}=W_\alpha$. Using
\eqref{hold} we conclude that $Y_{\alpha,\rho}(0)=0$ a.s.\
whenever $\rho>-\alpha$.

Convergence of the integrals in \eqref{defYpos} and
\eqref{defYneg} and a.s.\ continuity of the processes
$Y_{\alpha,\rho}$ will be proved in Lemma \ref{Ybeta} below. When
$W_\alpha$ is a Brownian motion (so that $\alpha=1/2-\varepsilon$
for any $\varepsilon\in (0,1/2)$), the process $Y_{\alpha,\rho}$
can be represented as a Skorokhod integral
\begin{equation}    \label{defY}
Y_{\alpha,\rho}(u):=\int_{[0,\,u]}(u-y)^\rho
\textrm{d}W_\alpha(y), \quad u \geq 0.
\end{equation}
The so defined process is called {\it Riemann–Liouville process}
or {\it fractionally integrated Brownian motion} with exponent
$\rho$ for $\rho>-1/2$. Since these processes appear for several
times in our presentation we reserve a special notation for them,
$R_\rho$. When $W_\alpha$ is a more general Gaussian process
satisfying the standing assumptions, the process $Y_{\alpha,\rho}$
may be called {\it fractionally integrated Gaussian process}. Note
that, for positive integer $\rho$, the process $Y_{\alpha,\rho}$
is up to a multiplicative constant an $r$-times integrated process
$W_{\alpha}$. This can be easily checked with the help of
integration by parts.

Throughout the paper we assume that the spaces $D$ and $D\times D$
are endowed with the $J_1$-topology and denote weak convergence in
these spaces by $\stackrel{J_{1}}\implies$ and
$\stackrel{J_1(D\times D)}\implies$, respectively. Comprehensive
information concerning the $J_1$-topology can be found in the
books \cite{Billingsley, Jacod+Shiryaev:2003}. In what follows we
use the notation $\mr^+:=[0,\infty)$.
\begin{theorem}\label{main}
Let $h \in D$ be an eventually nondecreasing function of bounded
variation which is regularly varying\footnote{This means that
$\lim_{t\to\infty}(h(tx)/h(t))=x^\beta$ for all $x>0$.} at
$\infty$ of index $\beta\geq 0$. Assume that
$\lim_{t\to\infty}h(t)=\infty$ when $\beta=0$ and that
\begin{equation}\label{basic}
\frac{N(t\cdot)-b(t\cdot)}{a(t)} \stackrel{J_{1}} \implies
W_{\alpha}(\cdot), \quad t \to \infty,
\end{equation}
where $a: \mr^+\to\mr^+$ is regularly varying at $\infty$ of
positive index, and $b:\mr^+\to\mr^+$ is a nondecreasing function.
Then
$$\frac{X(t\cdot)-\int_{[0,\,t \cdot]}h(t \cdot-y){\rm d} b(y)}{a(t) h(t)}
\stackrel{J_{1}} \implies Y_{\alpha,\beta}(\cdot),
\quad t \to \infty.$$
\end{theorem}
\begin{rem}
A perusal of the proof given below reveals that the assumption
$\lim_{t\to\infty}h(t)=\infty$ when $\beta=0$ is not needed if $h$
is nondecreasing on $\mr^+$ rather than eventually nondecreasing.
\end{rem}

Since $b$ is nondecreasing and $h$ is locally bounded and almost
everywhere continuous function (in view of $h \in D$), the
integral $\int_{[0,\,t]}h(t-y){\rm d} b(y)$ exists as a
Riemann-Stieltjes integral.

The remainder of the article is organized as follows. In Section
\ref{hold2} we investigate local H\"{o}lder continuity of the
limit processes in Theorem \ref{main}. In Section \ref{appl} we
give five specializations of Theorem \ref{main} for particular
sequences $(S_k)_{k\in\mn_0}$. Finally, we prove Theorem
\ref{main} in Section \ref{pro}.

\section{H\"{o}lder continuity of the limit processes}\label{hold2}

For the subsequent presentation, it is convenient to define the
process $W_\alpha$ on the whole line. To this end, put
$W_\alpha(x)=0$ for $x<0$. The right-hand side of \eqref{defYneg}
can then be given in an equivalent form
\begin{equation}\label{equiv}
Y_{\alpha,\rho}(u)= |\rho| \int_0^\infty (W_\alpha (u) -
W_\alpha(u-y)) y^{\rho-1}\textrm{d}y,\quad u>0.
\end{equation}
It is important for us that formula \eqref{hold} still holds true
for negative $x,y$. More precisely, we claim that, for all $T>0$,
all $-\infty<x,y\leq T$ and the same random variable $M_T$ as in
\eqref{hold},
\begin{equation}\label{sam23}
|W_\alpha(x)-W_\alpha(y)|\leq M_T|x-y|^\alpha.
\end{equation}
This inequality is trivially satisfied in the case $x\vee y\leq 0$
and follows from \eqref{hold} in the case $x\wedge y\geq 0$.
Assume now that $x\wedge y\leq 0<x\vee y$. Then
$|W_\alpha(x)-W_\alpha(y)|=|W_\alpha(x\vee y)|\leq M_T (x\vee
y)^\alpha\leq M_T|x-y|^\alpha$. Here, the first inequality is a
consequence of \eqref{hold} with $y=0$.
\begin{lemma} \label{Ybeta}
Let $\rho>-\alpha$. The following assertions hold:\newline 1)
$|Y_{\alpha,\rho}(u)|<\infty$ a.s.\ for each fixed $u>0$;\newline
2) the process $Y_{\alpha,\rho}$ is a.s.\ locally H\"{o}lder
continuous with exponent $\min(1, \, \alpha + \rho)$ if
$\alpha+\rho \neq 1$ and with arbitrary positive exponent less
than $1$ if $\alpha+\rho=1$; more precisely, in the latter
situation we have, for any $T^\ast>T$,
$$\sup_{0 \leq u \neq v \leq T}
\frac{|Y_{\alpha,\rho}(u)-Y_{\alpha,\rho}(v)|}{|u-v| \log (T^\ast |u-v|^{-1})} <
\infty \quad \text{{\rm a.s.}}$$
\end{lemma}

\begin{proof}
The case $\rho=0$ is trivial. Fix $T>0$.

\noindent {\sc Proof of 1)}. Using \eqref{hold} we obtain, for all
$u\in [0,T]$,
$$|Y_{\alpha,\rho}(u)|\leq \rho \int_0^u (u-y)^{\rho-1}|W_\alpha(y)|{\rm
d}y \leq M_T\rho \int_0^u (u-y)^{\rho-1}y^\alpha {\rm
d}y=M_T\rho{\rm B}(\rho,\alpha+1)u^{\rho+\alpha}< \infty\quad
\text{a.s.}$$ in the case $\rho>0$ and
\begin{eqnarray*}
|Y_{\alpha,\rho}(u)|
&\leq& u^\rho |W_\alpha(u)| + |\rho| \int_0^u |W_\alpha (u) - W_\alpha(u-y))
y^{\rho-1}\textrm{d}y \\
&\leq& M_T u^{\rho+\alpha} + |\rho|
M_T(\rho+\alpha)^{-1}u^{\rho+\alpha}= M_T \alpha (\rho +
\alpha)^{-1} u^{\rho+\alpha} < \infty\quad \text{a.s.}
\end{eqnarray*}
in the case $-\alpha<\rho<0$.

\noindent {\sc Proof of 2)}. By virtue of symmetry it is enough to
investigate the case $0\leq v<u\leq T$, and this is tacitly
assumed throughout the proof.

Assume first that $-\alpha < \rho < 0$. Appealing to \eqref{sam23}
and \eqref{hold} we conclude that, for $v>0$,
\begin{eqnarray*}
|\rho|^{-1}\left|Y_{\alpha,\rho}(u) - Y_{\alpha,\rho}(v)\right|
&=& \left| \int_0^\infty (W_\alpha(u) - W_\alpha(u-y) - W_\alpha(v) + W_\alpha(v-y))y^{\rho-1}\textrm{d}y \right| \\
&\leq& \int_0^{u-v} \left|W_\alpha(u) - W_\alpha(u-y)\right|y^{\rho-1}\textrm{d}y \\
&+& \int_0^{u-v} \left|W_\alpha(v) - W_\alpha(v-y)\right|y^{\rho-1}\textrm{d}y \\
&+& \int_{u-v}^\infty \left|W_\alpha(u) - W_\alpha(v)\right|y^{\rho-1}\textrm{d}y \\
&+& \int_{u-v}^\infty \left|W_\alpha(u-y) - W_\alpha(v-y)\right|y^{\rho-1}\textrm{d}y \\
&\leq& 2M_T \left( \int_0^{u-v} y^{\rho - 1 + \alpha}\textrm{d}y
+ (u-v)^\alpha \int_{u-v}^\infty y^{\rho - 1}\textrm{d}y\right) \\
&=& 2M_T \alpha (|\rho|(\rho+\alpha))^{-1}
(u-v)^{\rho+\alpha}\quad\text{a.s.}
\end{eqnarray*}
We already know from the proof of part 1) that a similar
inequality holds when $v=0$. Thus, the claim of part 2) has been
proved in the case $-\alpha<\rho<0$.

Assume now that $\rho \geq 1$. We infer with the help of
\eqref{hold} that
\begin{eqnarray*}
\left|Y_{\alpha,\rho}(u) - Y_{\alpha,\rho}(v)\right| &=&
\left|\rho \int_0^v((u-y)^{\rho-1}-(v-y)^{\rho-1})
W_{\alpha}(y)\textrm{d}y
+ \rho \int_v^u(u-y)^{\rho-1} W_{\alpha}(y)\textrm{d}y \right| \\
&\leq& \rho \int_0^v((u-y)^{\rho-1}-(v-y)^{\rho-1})
\left|W_{\alpha}(y)\right|\textrm{d}y
+ \rho \int_v^u(u-y)^{\rho-1} \left|W_{\alpha}(y)\right|\textrm{d}y \\
&\leq& \rho M_T \int_0^v((u-y)^{\rho-1}-(v-y)^{\rho-1}) y^\alpha
\textrm{d}y
+ \rho M_T \int_v^u(u-y)^{\rho-1} y^\alpha \textrm{d}y  \\
&=& \rho M_T \int_0^u (u-y)^{\rho-1} y^\alpha \textrm{d}y
- \rho M_T \int_0^v (v-y)^{\rho-1} y^\alpha \textrm{d}y \\
&=& \rho M_T {\rm B}(\rho,\alpha+1) (u^{\rho + \alpha}-v^{\rho +
\alpha})\\&\leq& \rho (\rho+\alpha)M_T {\rm
B}(\rho,\alpha+1)T^{\rho+\alpha-1}(u-v) \quad\text{a.s.},
\end{eqnarray*}
where the last inequality follows from the mean value theorem for
differentiable functions.

It remains to investigate the case $0 < \rho < 1$. We shall use
the following decomposition
\begin{eqnarray*}
\left|Y_{\alpha,\rho}(u) - Y_{\alpha,\rho}(v)\right| &=&
\left|\rho \int_0^u(u-y)^{\rho-1} W_{\alpha}(y)\textrm{d}y
- \rho \int_0^v(v-y)^{\rho-1} W_{\alpha}(y)\textrm{d}y \right| \\
&=& \left | \rho \int_0^v (W_{\alpha}(v) - W_{\alpha}(v-y))(y^{\rho-1}-(y+u-v)^{\rho-1})\textrm{d}y \right. \\
&-& \left. \rho \int_0^{u-v}
(W_\alpha(v)-W_{\alpha}(u-y))y^{\rho-1}\textrm{d}y
+ W_\alpha(v)(u^\rho - v^\rho) \right| \\
&\leq& I_1 + I_2+I_3,
\end{eqnarray*}
where $$I_1:=\rho \int_0^v |W_{\alpha}(v) -
W_{\alpha}(v-y)|(y^{\rho-1}-(y+u-v)^{\rho-1})\textrm{d}y,$$
$$I_2:= \rho \int_0^{u-v}
|W_\alpha(v)-W_{\alpha}(u-y)|y^{\rho-1}\textrm{d}y\quad\text{та}\quad
I_3:=|W_\alpha(v)|(u^\rho - v^\rho).$$ The summand $I_1$ can be
estimated as follows
\begin{eqnarray*}
I_1 &\leq& \rho M_T \int_0^v y^\alpha(y^{\rho-1}-(y+u-v)^{\rho-1})\textrm{d}y \\
&=& \rho M_T (u-v)^{\alpha + \rho} \int_0^{v/(u-v)}
t^\alpha(t^{\rho-1}-(t+1)^{\rho-1})\textrm{d}t.
\end{eqnarray*}
Using the inequality $x^{\rho-1} - (x+1)^{\rho-1} \leq
(1-\rho)x^{\rho-2}$ for $x>0$ gives
\begin{eqnarray*}
I_1&\leq& \rho (1-\rho)M_T (u-v)^{\alpha + \rho} \int_0^{v/(u-v)}
t^{\alpha+\rho-2}{\rm d}t \\&=& \rho
(1-\rho)(\alpha+\rho-1)^{-1}M_T v^{\alpha+\rho-1} (u-v)\\&\leq&
\rho (1-\rho)(\alpha+\rho-1)^{-1}M_T T^{\alpha+\rho-1} (u-v) in
\end{eqnarray*}
in the case $\alpha+\rho>1$ and
\begin{eqnarray*}
I_1 &\leq& \rho M_T(u-v)^{\alpha+\rho} \left(\int_0^1
t^\alpha(t^{\rho-1}-(t+1)^{\rho-1})\textrm{d}t
+ \int_1^\infty t^\alpha(t^{\rho-1}-(t+1)^{\rho-1})\textrm{d}t\right) \\
&\leq& \rho M_T(u-v)^{\alpha+\rho} \left(\int_0^1 t^{\alpha +
\rho-1}\textrm{d}t+(1 - \rho) \int_1^\infty t^{\alpha + \rho-2} \textrm{d}t\right) \\
&=& \rho M_T \left(\frac{1}{\alpha+\rho} +
\frac{1-\rho}{1-\alpha-\rho}\right) (u-v)^{\alpha+\rho}
\end{eqnarray*}
in the case $0<\alpha+\rho<1$. Also,
\begin{eqnarray*}
I_1 &\leq& \rho M_T(u-v) \int_0^{v/(u-v)} \left(1-\left(\frac{t}{t+1}\right)^\alpha \right)\textrm{d}t \\
&\leq& \rho M_T(u-v) \int_0^{v/(u-v)} \frac{\textrm{d}t}{t+1} \\
&=& \rho M_T(u-v) \ln \frac{u}{u-v} \\
&\leq& \rho M_T(u-v) \ln \frac{T}{u-v}
\end{eqnarray*}
in the case $\alpha+\rho=1$.

Further, $$I_2\leq \rho M_T \int_0^{u-v}(u-v-y)^\alpha
y^{\rho-1}{\rm d}y=\rho M_T {\rm
B}(\rho,\alpha+1)(u-v)^{\alpha+\rho}.$$ In the case
$\alpha+\rho>1$ the inequality $(u-v)^{\alpha+\rho}\leq
T^{\alpha+\rho-1}(u-v)$ has to be additionally used.

Finally, $$I_3\leq M_T v^\alpha(u^\rho-v^\rho)\leq M_T
(u^{\alpha+\rho}-v^{\alpha+\rho}).$$ The right-hand side does not
exceed $M_T(u-v)^{\alpha+\rho}$ in the case $0<\alpha+\rho\leq 1$
in view of subadditivity of $x\mapsto x^{\alpha+\rho}$ on
$[0,\infty)$ and $M_T(\alpha+\rho)T^{\alpha+\rho-1}(u-v)$ in the
case $\alpha+\rho>1$.

The proof of Lemma \ref{Ybeta} is complete.
\end{proof}

\section{Applications of Theorem \ref{main}}\label{appl}

In this section we give five examples of particular sequences
$(S_k)_{k\in\mn_0}$ which satisfy limit relation \eqref{basic}
with four different Gaussian processes $W_\alpha$. Throughout the
section we always assume, without further notice, that $h$
satisfies the assumptions of Theorem \ref{main}.

In the case where the sequence $(S_k)_{k\in\mn_0}$ is a.s.\
nondecreasing, the counting process $(N(t))_{t\geq 0}$ is nothing
else but a generalized inverse function for $(S_k)$, that is,
\begin{equation}\label{inv}
N(t)=\inf\{k\in\mn: S_k>t\}\quad\text{a.s.,}\quad t\geq 0.
\end{equation}
In view of this, if a functional limit theorem for $S_{[ut]}$ in
the $J_1$-topology on $D$ holds, and the limit process is a.s.\
continuous, then the corresponding functional limit theorem for
$N(ut)$ in the $J_1$-topology on $D$ is a simple consequence. A
detailed discussion of this fact can be found, for instance, in
\cite{Iglehart+Whitt:1971}. If the sequence $(S_k)_{k\in\mn_0}$ is
not monotone (as, for instance, at point 2 below), then equality
\eqref{inv} is no longer true, and the proof of a functional limit
theorem for $N(ut)$ requires an additional specific argument in
every particular case.

\noindent 1. {\sc Delayed standard random walk.} Let $\xi_1$,
$\xi_2,\ldots$ be i.i.d.\ nonnegative random variables which are
independent of a random variable $S_0$. The possibility that
$S_0=0$ a.s.\ is not excluded. The random sequence
$(S_k)_{k\in\mn_0}$ defined by $S_k:=S_0+\xi_1+\ldots+\xi_k$ for
$k\in\mn_0$ is called {\it delayed standard random walk}. In the
case $S_0=0$ a.s.\ the term {\it zero-delayed standard random
walk} is used. Denote by $(N_0(t))_{t\geq 0}$ the counting process
for a zero-delayed standard random walk. It is well-known (see,
for instance, Theorem 1b(i) in \cite{Bingham:1973}) that

\noindent a) if $\sigma^2:={\rm Var}\,\xi_1\in (0,\infty)$, then
\begin{equation}\label{rela}
\frac{N_0(t\cdot)-\mu^{-1}t(\cdot)}{(\sigma^2\mu^{-3}t)^{1/2}}
\stackrel{J_{1}} \implies B(\cdot) , \quad t \to \infty,
\end{equation}
where $\mu:=\me \xi_1<\infty$, and $(B(u))_{u\geq 0}$ is a
standard Brownian motion (so that relation \eqref{basic} holds
with $b(t)=\mu^{-1}t$ and $a(t)=(\sigma^2\mu^{-3}t)^{1/2}$);

\noindent b) if
\begin{equation} \label{assu2}
\sigma^2=\infty~~\text{and}~~\int_{[0,\,x]}y^2\mmp\{\xi_1\in {\rm
d}y\}\sim L(x),\quad x\to\infty
\end{equation}
for some $L$ slowly varying at $\infty$, then
\begin{equation}\label{rela1}
\frac{N_0(t\cdot)-\mu^{-1}t(\cdot)}{\mu^{-3/2}c(t)}
\stackrel{J_{1}} \implies B(\cdot) , \quad t \to \infty,
\end{equation}
where $c$ is a positive measurable function satisfying
$\lim_{t\to\infty}c(t)^{-2}tL(c(t))=1$ (so that relation
\eqref{basic} holds with $b(t)=\mu^{-1}t$ and
$a(t)=\mu^{-3/2}c(t)$; since $c$ is asymptotically inverse for
$t\mapsto t^2/L(t)$, an application of Proposition 1.5.15 in
\cite{Bingham:1989} enables us to conclude that $c$, hence also
$a$ are regularly varying at $\infty$ of index $1/2$).

The counting process $(N(t))_{t\geq 0}$ for a delayed standard
random walk satisfies the same functional limit theorems which
follows from the limit relation: for all $u\in [0,T]$
$$a(t)^{-1}\sup_{u\in [0,\,T]}(N(tu)-N_0(tu)) \stackrel{\rm{\mathbb{P}}}{\longrightarrow} 0,
\quad t \to \infty,$$ where, depending on the case, either
$a(t)=(\sigma^2\mu^{-3}t)^{1/2}$ or $a(t)=\mu^{-3/2}c(t)$, and
$(N_0(t))_{t\in\mr}$ is the counting process for the corresponding
zero-delayed standard random walk (of course, $N_0(t)=0$ for
$t<0$). The last centered formula is a consequence of the equality
$N(t)=N_0(t-S_0)$, $t\geq 0$, the relation obtained in Lemma A.1
of \cite{Iksanov:2013}
$$a(t)^{-1} \sup_{u \in [0,\,T]}{(N_0(ut)-N_0(ut-h))}
\stackrel{\rm{\mathbb{P}}}{\longrightarrow} 0, \quad t \to
\infty$$ which holds for any positive $h$ and $T$, and the fact
that $S_0$ is independent of $(N_0(t))_{t\in\mr}$. Thus, according
to Theorem \ref{main}, we have both for delayed and zero-delayed
standard random walks
$$\frac{X(t\cdot)-\mu^{-1}\int_0^{t \cdot}{h(y){\textrm
d}y}}{(\sigma^2\mu^{-3}t)^{1/2} h(t)} \stackrel{J_{1}} \implies
\beta
\int_0^{(\cdot)}(\cdot-y)^{\beta-1}B(y)\textrm{d}y=R_\beta(\cdot),\quad
t \to \infty,$$ provided that $\sigma^2\in (0,\infty)$ (in the
case $\beta=0$, the limit process is $R_0=B$), and
$$\frac{X(t\cdot)-\mu^{-1}\int_0^{t \cdot}{h(y){\textrm
d}y}}{\mu^{-3/2}c(t)h(t)} \stackrel{J_{1}} \implies
R_\beta(\cdot), \quad t \to \infty$$ provided that conditions
\eqref{assu2} hold. In particular, irrespective of whether the
variance is finite or not the limit process is a fractionally
integrated Brownian motion with parameter $\beta$. As far as
zero-delayed standard random walks are concerned, the
aforementioned results can be found in Theorem 1.1 (A1, A2) of
\cite{Iksanov:2013}.

\noindent 2. {\sc Perturbed random walks}. Let $(\xi_1, \eta_1)$,
$(\xi_2,\eta_2)\ldots$ be i.i.d.\ random vectors with nonnegative
coordinates. Put $$S_1:=\eta_1,\quad
S_n:=\xi_1+\ldots+\xi_{n-1}+\eta_n,\quad n\geq 2.$$ The so defined
sequence $(S_n)_{n\in\mn}$ is called {\it perturbed random walk}.
Various properties of perturbed random walks are discussed in the
monograph \cite{Iksanov:2016}.

Assume that $\sigma^2={\rm Var}\,\xi_1\in (0,\infty)$ and $\me
\eta^a<\infty$ for some $a>0$. Put $F(x):=\mmp\{\eta_1\leq x\}$
for $x\in\mr$. According to Theorem 3.2 in
\cite{Alsmeyer+Iksanov+Marynych:2017},
$$\frac{N(t\cdot)-\mu^{-1}\int_0^{t\cdot}F(y){\rm d}y}{\sqrt{\sigma^2\mu^{-3}t}}
\stackrel{J_{1}} \implies B(\cdot), \quad t \to \infty,$$ where
$\mu=\me\xi_1<\infty$. Therefore, by Theorem \ref{main},
$$\frac{X(t\cdot)-\mu^{-1}\int_0^{t \cdot} h(y)F(y){\rm
d}y}{(\sigma^2\mu^{-3}t)^{1/2} h(t)} \stackrel{J_{1}}
\implies R_\beta(\cdot),\quad t \to \infty.$$

\noindent 3. {\sc Random walks with long memory.} Let $\xi_1$,
$\xi_2,\ldots$ be i.i.d.\ positive random variables with finite
mean. Assume that these are independent of random variables
$\theta_1$, $\theta_2,\ldots$ which form a centered stationary
Gaussian sequence with $\me \theta_1\theta_{k+1}\sim
k^{2d-1}\ell(k)$ as $k\to\infty$ for some $d\in (0,1/2)$. Put
$S_0:=0$ and
$$S_n-S_{n-1}=\xi_n e^{\theta_n},\quad n\in\mn.$$

Recall that a fractional Brownian motion with Hurst index $H\in
(0,1)$ is a centered Gaussian process $B_H:=(B_H(u))_{u\geq 0}$
with covariance $\me
B_H(u)B_H(v)=2^{-1}(u^{2H}+v^{2H}-(u-v)^{2H})$ for $u,v\geq 0$.
This process has stationary increments and is self-similar of
index $H$. Therefore, for any $p>0$, $$\me
|B_H(u)-B_H(v)|^p=(u-v)^{Hp}\me |B_H(1)|^p,\quad u,v\geq 0.$$
According to the Kolmogorov-Chentsov sufficient conditions, there
exists a version of $B_H$ (which we also denote by $B_H$) which is
a.s.\ H\"{o}lder continuous with exponent smaller than $H-1/p$ for
any $p>0$, hence also, smaller than $H$.

According to Example 4.25 on p.~357 in \cite{Beran et al:2013},
$$\frac{N(t\cdot)-m_1^{-1}t(\cdot)}{(d(2d+1))^{-1/2}m_1^{-3/2-d}m_2
t^{d+1/2}(\ell(t))^{1/2}} \stackrel{J_{1}} \implies
B_{d+1/2}(\cdot),\quad t\to\infty,$$ where $m_1:=\me S_1=\me
\xi_1\me e^{\theta_1}$ and $m_2:=\me\xi_1\me
\theta_1e^{\theta_1}$. An application of Theorem \ref{main} yields
$$\frac{X(t\cdot)-m_1^{-1}\int_0^{t \cdot}{h(y){\textrm
d}y}}{(d(2d+1))^{-1/2}m_1^{-3/2-d}m_2
t^{d+1/2}(\ell(t))^{1/2}h(t)} \stackrel{J_{1}} \implies \beta
\int_0^{(\cdot)}(\cdot-y)^{\beta-1}B_{d+1/2}(y)\textrm{d}y,\quad t
\to \infty$$ if $\beta>0$. If $\beta=0$, the limit process is
$B_{d+1/2}$.

\noindent 4. {\sc Counting process in a branching random walk}.
Assume that the random variables $\xi_k$ defined at point 1 are
a.s.\ positive. Denote by $(N^\prime(t))_{t\geq
0}:=(N_0(t)-1)_{t\geq 0}$ the corresponding renewal process. For
some integer $k\geq 2$, we take in the role of $N(t)$ the number
of the $k$th generation individuals with positions $\leq t$ in a
branching random walk in which the first generation individuals
are located at the points $S_1$, $S_2,\ldots$ (a more precise
definition can be found in Section 1.2 of
\cite{Iksanov+Kabluchko:2018}).

Assume that $\sigma^2={\rm Var}\,\xi_1 \in (0,\infty)$. Theorem
1.3 in \cite{Iksanov+Kabluchko:2018} implies that
$$\frac{N(t\cdot)-(t\cdot)^k/(k!\mu^k)}{((k-1)!)^{-1}\sqrt{\sigma^2\mu^{-2k-1}t^
{2k-1}}} \stackrel{J_{1}} \implies R_{k-1}(\cdot), \quad t \to
\infty,$$ where $\mu=\me\xi_1<\infty$. Of course, for $k=1$ this
limit relation is also valid and amounts to \eqref{rela} as it
must be. By Lemma \ref{Ybeta}, the process $R_{k-1}$ is a.s.\
locally H\"{o}lder continuous with any positive exponent smaller
than $k-1/2$. Thus, Theorem \ref{main} applies and gives
\begin{eqnarray*}
\frac{X(t\cdot)-((k-1)!\mu^k)^{-1}\int_0^{t
\cdot}{h(t\cdot-y)y^{k-1}{\textrm
d}y}}{((k-1)!)^{-1}\sqrt{\sigma^2\mu^{-2k-1}t^{2k-1}}
h(t)}&\stackrel{J_{1}} \implies& \beta \int_0^{(\cdot)}
(\cdot-y)^{\beta-1}R_{k-1}(y)\textrm{d}y\\&=&(k-1){\rm
B}(k-1,\beta+1) R_{\beta+k-1}(\cdot),\quad t \to \infty,
\end{eqnarray*}
where ${\rm B}(\cdot,\cdot)$ is the beta function. The latter
equality can be checked as follows: for $u>0$
\begin{eqnarray*}
\beta \int_0^u (u-y)^{\beta-1}R_{k-1}(y)\textrm{d}y
&=&\beta \int_0^u (u-y)^{\beta-1}(k-1) \int_0^y (y-x)^{k-2}B(x)\textrm{d}x
\textrm{d}y \\
&=&\beta(k-1) \int_0^u B(x) \int_x^ u (u-y)^{\beta-1}(y-x)^{k-2}\textrm{d}y
\textrm{d}x \\
&=&\beta(k-1) \int_0^u B(x) \int_0^{u-x}(u-x-y)^{\beta-1}y^{k-2}\textrm{d}y
\textrm{d}x \\
&=&\beta(k-1){\rm B}(k-1, \beta) \int_0^u (u-x)^{\beta+k-2} B(x) \textrm{d}x \\
&=& \frac{\beta(k-1){\rm B}(k-1, \beta)}{\beta+k-1}
R_{\beta+k-1}(u)\\&=&(k-1){\rm B}(k-1,\beta+1)R_{\beta+k-1}(u).
\end{eqnarray*}

\noindent 5. {\sc Inhomogeneous Poisson process.} Let
$(N(t))_{t\geq 0}$ be an inhomogeneous Poisson process with $\me
N(t)=m(t)$ for a nondecreasing function $m(t)$ satisfying
\begin{equation}\label{asym}
m(t)~\sim~ c t^w,\quad t\to\infty,
\end{equation}
where $c,w>0$. Without loss of generality we can identify
$(N(t))_{t\geq 0}$ with the process $(N^\ast(m(t)))_{t\geq 0}$,
where $(N^\ast(t))_{t\geq 0}$ is a homogeneous Poisson process
with $\me N^\ast(t)=t$, $t\geq 0$. According to \eqref{rela},
\begin{equation}\label{intermed}
\frac{N^\ast (t\cdot)-(t \cdot)}{t^{1/2}} \stackrel{J_{1}}
\implies B(\cdot) , \quad t \to \infty.
\end{equation}

Dini's theorem in combination with \eqref{asym} ensures that
\begin{equation}\label{dini}
\lim_{t\to\infty}\sup_{u\in
[0,\,T]}\Big|\frac{m(tu)}{ct^w}-u^w\Big|=0
\end{equation}
for all $T>0$. It is known (see, for instance, Lemma 2.3 on p.~159
in \cite{Gut:2009}) that the composition mapping
$(x,\varphi)\mapsto (x\circ \varphi)$ is continuous on continuous
functions $x:\mr_+ \to \mr$ and continuous nondecreasing functions
$\varphi: \mr_+\to \mr_+$. Using this fact in conjunction with
\eqref{intermed}, \eqref{dini} and continuous mapping theorem we
infer
\begin{equation}\label{intermed2}
\frac{N(t\cdot)-m(t\cdot)}{(ct^w)^{1/2}}\stackrel{J_{1}} \implies
B((\cdot)^w),\quad t\to\infty.
\end{equation}
Thus, the limit process is a time-changed Brownian motion. An
application of Theorem \ref{main} yields
$$\frac{X(t\cdot)-\mu^{-1}\int_0^{t \cdot} h(y){\rm d}m(y)}{(ct^w)^{1/2}h(t)}
\stackrel{J_{1}} \implies
\beta\int_0^{(\cdot)}(\cdot-y)^{\beta-1}B(y^w){\rm d}y,\quad t \to
\infty$$ if $\beta>0$. If $\beta=0$, the limit process is
$B((\cdot)^w)$.

\section{Proof of Theorem \ref{main}}\label{pro}

To prove weak convergence of finite-dimensional distributions we
need an auxiliary lemma.
\begin{lemma} \label{EZuZv}
Let $f:\mr^+ \times \mr^+ \to (0,\infty)$ be a function which is
nondecreasing in the second coordinate and satisfies $\lim_{t \to
\infty}f(t,x) =x^{\beta}$ for all $x>0$ and some $\beta \geq 0$.
For $t,x\geq 0$, set
$$Z(t,x):=\int_{[0,\,x]}W_\alpha(x-y) {\rm d}_y f(t,y),$$
$$Z(x):=\int_{[0,\,x]}W_\alpha(x-y) {\rm d}y^\beta \text{ for }
\beta > 0 \text{ and } Z(x):= W_\alpha(x) \text{ for } \beta =
0.$$ Then, for any $u,v > 0$, $$\lim_{t \to \infty} \me
Z(t,u)Z(t,v) = \me Z(u)Z(v).$$
\end{lemma}
\begin{proof}
Fix any $u,v>0$. For each $t>0$, denote by $Q_t^{(u)}$ and
$Q_t^{(v)}$ independent random variables with the distribution
functions
$$\mmp\{Q_t^{(u)} \leq y\} =
\begin{cases}
0, \quad \text{if } y<0, \\
\frac{f(t,y)}{f(t,u)}, \quad \text{if } y \in [0,u],\\
1, \quad \text{if } y>u
\end{cases}
\text{and} \quad \mmp\{Q_t^{(v)} \leq y\} =
\begin{cases}
0, \quad \text{if } y<0, \\
\frac{f(t,y)}{f(t,v)}, \quad \text{if } y \in [0,v],\\
1, \quad \text{if } y>v.
\end{cases}
$$
Also, denote by $Q^{(u)}$ and $Q^{(v)}$ independent random
variables with the distribution functions
$$\mmp\{Q^{(u)} \leq y\} =
\begin{cases}
0, \quad \text{if } y<0, \\
(\frac{y}{u})^{\beta}, \quad \text{if } y \in [0,u],\\
1, \quad \text{if } y>u
\end{cases}
\text{and} \quad \mmp\{Q^{(v)} \leq y\} =
\begin{cases}
0, \quad \text{if } y<0, \\
(\frac{y}{v})^{\beta}, \quad \text{if } y \in [0,v],\\
1, \quad \text{if } y>v.
\end{cases}
$$
By assumption, $$(Q_t^{(u)}, Q_t^{(v)}) \stackrel{{\rm d}} \to
(Q^{(u)}, Q^{(v)}), \quad t \to \infty.$$

Define the function $r(x,y):=\me W_\alpha(x)W_\alpha(y)$ on
$\mr^+\times \mr^+$. Using a.s.\ continuity of $W_\alpha$,
Lebesgue's dominated convergence theorem and the fact that,
according to Theorem 3.2 on p.~ 63 in \cite{Adler}, $\me(\sup_{z
\in [0,\,T]}W_\alpha(z))^2<\infty$ we conclude that $r$ is
continuous, hence also bounded on $[0,T] \times [0,T]$ for all
$T>0$. This entails
$$ r(u-Q_t^{(u)}, v-Q_t^{(v)}) \stackrel{{\rm d}}{\longrightarrow} r(u-Q^{(u)}, v-Q^{(v)}),\quad t \to \infty$$
and thereupon $$\lim_{t \to \infty} \me r(u-Q_t^{(u)},
v-Q_t^{(v)}) = \me r(u-Q^{(u)}, v-Q^{(v)})$$ by Lebesgue's
dominated convergence theorem. Further,
\begin{eqnarray*}
\me Z(t,u)Z(t,v)
&=&f(t,u)f(t,v)\int_{[0,\,u]}\int_{[0,\,v]} \me W_\alpha(u-y)W_\alpha(v-z) \textrm{d}_y \left(\frac{f(t,y)}{f(t,u)}\right) \textrm{d}_z \left(\frac{f(t,z)}{f(t,v)}\right) \\
&=& f(t,u)f(t,v)\me r \left(u-Q_t^{(u)}, v-Q_t^{(v)}\right)
\underset{t \to \infty}{\longrightarrow} (uv)^{\beta} \me r
\left(u-Q^{(u)}, v-Q^{(v)}\right).
\end{eqnarray*}
It remains to note that while in the case $\beta > 0$ we have
\begin{eqnarray*}
&&(uv)^{\beta} \me r (u-Q^{(u)}, v-Q^{(v)})\\&=& \int_{[0,\,u]}\int_{[0,\,v]} r(u-y, v-z) \textrm{d}_y \left(u^{\beta} \rm{\mathbb{P}}\{Q^{(u)} \leq y\} \right) \textrm{d}_z \left(v^{\beta} \rm{\mathbb{P}}\{Q^{(v)} \leq z\}\right) \\
&=& \int_{[0,\,u]}\int_{[0,\,v]}r(u-y, v-z) \textrm{d} y^{\beta}
\textrm{d} z^{\beta}=\me \int_{[0,\,u]}W_\alpha(u-y) \textrm{d} y^{\beta} \int_{[0,\,v]}W_\alpha(v-z) \textrm{d} z^{\beta} \\
&=& \me Z(u)Z(v),
\end{eqnarray*}
in the case $\beta = 0$ we have
$$(uv)^{\beta} \me r (u-Q^{(u)}, v-Q^{(v)}) = r(u,v) = \me W_\alpha(u)W_\alpha(v) = \me Z(u)Z(v).$$
The proof of Lemma \ref{EZuZv} is complete.
\end{proof}

\begin{proof}[Proof of Theorem \ref{main}]
Since $h$ is eventually nondecreasing, there exists $t_0>0$ such
that $h(t)$ is nondecreasing for $t>t_0$. Being a regularly
varying function of nonnegative index, $h$ is eventually positive.
Hence, increasing $t_0$ if needed we can ensure that $h(t)>0$ for
$t>t_0$. We first show that the behavior of the function of
bounded variation $h$ on $[0,t_0]$ does not affect weak
convergence of the general shot noise process. Once this is done,
we can assume, without loss of generality, that $h(0)=0$ and that
$h$ is nondecreasing on $\mr^+$.

Integrating by parts yields
\begin{eqnarray}
X(tu)-\int_{[0,\,tu]}(h(tu-y){\rm d}
b(y)&=&\int_{[0,\,u]}(h(t(u-y)){\rm
d}_y(N(ty)-b(ty))\notag\\&=&(h(tu)-h((tu)-))(N(0)-b(0))\notag\\&+&\int_{(0,\,u]}(N(ty)-b(ty))){\rm
d}_y(-h(t(u-y)))\label{aux2}
\end{eqnarray}
For all $T>0$, $$\frac{\sup_{u\in
[0,\,T]}|h(tu)-h((tu)-)||N(0)-b(0)|}{a(t)h(t)}\leq
\frac{h(tT)}{h(t)}\frac{|N(0)-b(0)|}{a(t)} \tp 0,\quad
t\to\infty$$ because $h$ is regularly varying at $\infty$. Denote
by $\mathbb{V}_0^{t_0}(h)$ the total variation of $h$ on
$[0,t_0]$. By assumption, $\mathbb{V}_0^{t_0}(h)<\infty$. For all
$T>0$,
$$\frac{\sup_{u\in [0,\,T]}\int_{[u-t_0/t,\,u)}(N(ty)-b(ty)){\rm
d}_y(-h(t(u-y)))}{a(t)h(t)}\leq \frac{\sup_{u\in
[0,\,T]}|N(tu)-b(tu)|}{a(t)}\frac{\mathbb{V}_0^{t_0}(h)}{h(t)}\tp
0$$ as $t\to\infty$. The convergence to $0$ is justified by the
facts that, according to \eqref{basic}, the first factor converges
in distribution to $\sup_{u\in [0,\,T]}|W_\alpha(u)|$, whereas the
second trivially converges to $0$. Recall that, when $\beta=0$,
$\lim_{t\to\infty}h(t)=\infty$ holds by assumption, whereas, when
$\beta>0$, it holds automatically. Thus, as was claimed, while
investigating the asymptotic behavior of the second summand on the
right-hand side of \eqref{aux2} we can and do assume that $h(0)=0$
and that $h$ is nondecreasing on $\mr^+$.

Skorokhod's representation theorem ensures that there exist
versions $(\hat N(t))_{t \geq 0}$ and $(\hat W_\alpha(t))_{t \geq
0}$ of the processes $(N(t))_{t \geq 0}$ and $(W_\alpha(t))_{t
\geq 0}$ such that, for all $T>0$,
\begin{equation}\label{unif}
\lim_{t \to \infty} \sup_{u \in [0,\,T]} |\hat
W_\alpha^{(t)}(u)-\hat W_\alpha(u)|=0 \quad \text{a.s.},
\end{equation}
where $\hat W_\alpha^{(t)}(u):=\frac{\hat N(tu)-b(tu)}{a(t)}$ for
$t>0$ and $u\geq 0$. For each $t>0$, set $h_t(x):=h(tx)/h(t)$,
$x\geq 0$,
$$X_t(u):=(a(t))^{-1}\int_{(0,\,u]}(N(ty)-b(ty))){\rm
d}_y(-h_t(u-y)),\quad u\geq 0$$ and
\begin{eqnarray*}
\hat X^\ast_t(u):=
\int_{(0,\,u]} \hat W_\alpha^{(t)}(y)\textrm{d}_y
(-h_t(u-y)),\quad u\geq 0.
\end{eqnarray*}
The distributions of the processes $(X_t(u))_{u\geq 0}$ and $(\hat
X_t^\ast(u))_{u\geq 0}$ are the same. Hence, it remains to check
that
\begin{equation}\label{thConv1}
\lim_{t \to \infty} \int_{(0,\,u]} (\hat W_\alpha^{(t)}(y)- \hat
W_\alpha(y))\textrm{d}_y (-h_t(u-y))= 0\quad \text{a.s.}
\end{equation}
in the $J_1$-topology on $D$ and
\begin{equation}    \label{thConv2}
\int_{(0,\,\cdot]} W_\alpha(y)\textrm{d} (-h_t(\cdot-y))
\stackrel{J_{1}} \implies Y_{\alpha,\beta}(\cdot), \quad t \to
\infty.
\end{equation}
In view of \eqref{unif} and monotonicity of $h_t$, we have, for
all $T>0$ as $t \to \infty$,
\begin{eqnarray*}
&& \sup_{u \in [0,\,T]} \left|\int_{(0,\,u]} \left(\hat W_\alpha^{(t)}(u)-\hat W_\alpha(u)\right)\textrm{d}_y (-h_t(u-y))\right| \\
&\leq& \sup_{u \in [0,\,T]} \left|\hat W_\alpha^{(t)}(u)-\hat
W_\alpha(u)\right| h_t(T) \longrightarrow 0\quad \text{a.s.}
\end{eqnarray*}
which proves \eqref{thConv1}.

Since $W_\alpha$ is a Gaussian process, the convergence of the
finite-dimensional distributions in \eqref{thConv2} which is
equivalent to the convergence of covariances follows from Lemma
\ref{EZuZv}. While applying the lemma we use the equalities
$$Y_{\alpha,\,\beta}(u)=\int_{(0,\,u]} W_\alpha(y)\textrm{d}_y
(-(u-y)^\beta)$$ when $\beta>0$ and
$Y_{\alpha,\,\beta}(u)=W_\alpha(u)$ when $\beta=0$. Our next step
is to prove tightness on $D[0,T]$, for all $T>0$, of
$$\hat{X}_t(u):=\int_{(0,\,u]} W_\alpha(y) \textrm{d}_y (-h_t(u-y))=\int_{[0,\,u)} W_\alpha(u-y) \textrm{d}_y h_t(y),
\quad u \geq 0.$$ By Theorem 15.5 in \cite{Billingsley}, it is
enough to show that, for any $r_1,r_2>0$, there exist
$t_0,\delta>0$ such that, for all $t \geq t_0$,
\begin{equation}    \label{IneqTightness}
{\rm{\mathbb{P}}}\{\sup_{0 \leq u,v\leq T, |u-v|\leq \delta}
|\hat{X}_t(u)-\hat{X}_t(v)| > r_1\} \leq r_2.
\end{equation}
Put $l:=\max(u,v)$. Recalling that $W_\alpha(s)=0$ for $s<0$ (see
the beginning of Section \ref{hold2}) we have, for $0\leq u,v\leq
T$ and $|u-v|\leq \delta$,
\begin{eqnarray*}
|\hat{X}_t(u)-\hat{X}_t(v)|
&=& \left| \int_{[0,\,l)}( W_\alpha(u-y)-W_\alpha(v-y)) \textrm{d}_y h_t(y) \right| \\
&\leq& M_T |u-v|^\alpha h_t(T) \leq M_T |u-v|^{\alpha}\lambda
\end{eqnarray*}
for large enough $t$ and a positive constant $\lambda$. The
existence of $\lambda$ is justified by the relation $\lim_{t
\to\infty}h_t(T)=T^\beta<\infty$. Decreasing $\delta$ if needed,
we ensure that inequality \eqref{IneqTightness} holds for any
positive $r_1$ and $r_2$. The proof of Theorem \ref{main} is
complete.
\end{proof}

\end{document}